\documentclass[11pt]{article}
\usepackage{amsmath,amsthm,amssymb,amsfonts}
\usepackage{float}
\usepackage{color,fullpage}
\usepackage{cite}
\usepackage[title,titletoc]{appendix}
\usepackage{algorithm,algorithmic}
\usepackage{tikz}
\usepackage[overload]{empheq}
\UseRawInputEncoding
\usepackage{amsthm}
\newcommand{\RR}{\mathbb{R}}

\newcommand{\EE}{\mathbb{E}}
\newcommand{\dom}{{\mathrm{dom}}\,} 

\newcommand{\dist}{{\mathrm{dist}}}

\newcommand{\cX}{{\mathcal{X}}}

\newcommand{\rint}{{{\rm int}\,}}

\DeclareMathOperator*{\Min}{minimize\quad}

\newcommand{\st}{\mbox{subject to}}
\newtheorem{assumption}{Assumption}[section]

\newtheorem{theorem}{Theorem}[section]
\newtheorem{lemma}{Lemma}[section]
\newtheorem{definition}{Definition}[section]

\newtheorem{corollary}{Corollary}[section]

\makeatletter

\@addtoreset{equation}{section} \makeatother

\begin{document}

\title{Revisiting Linearized Bregman Iterations under Lipschitz-like Convexity Condition}

\author{Hui Zhang\thanks{
Corresponding author. Department of Mathematics, National University of Defense Technology,
Changsha, Hunan 410073, China.  Email: \texttt{h.zhang1984@163.com}
}
\and Lu Zhang\thanks{Department of Mathematics, National University of Defense Technology,
Changsha, Hunan 410073, China.}
\and Hao-Xing Yang\thanks{Department of Mathematics, National University of Defense Technology,
Changsha, Hunan 410073, China.}
}

\date{\today}

\maketitle

\begin{abstract}
The linearized Bregman iterations (LBreI) and its variants have received considerable attention in signal/image processing and compressed sensing. Recently, LBreI has been extended to a larger class of nonconvex functions, along with several theoretical issues left for further investigation. In particular, the gradient Lipschitz continuity assumption precludes its use in many practical applications. In this study, we propose a generalized algorithmic framework to unify LBreI-type methods. Our main discovery is that the gradient Lipschitz continuity assumption can be replaced by a Lipschitz-like convexity condition in both convex and nonconvex cases. The proposed framework and theory are then applied to linear/quadratic inverse problems.
\end{abstract}




\section{Introduction}
Let $x_o\in\RR^n$ be an unknown vector which stands for a signal or an image. Assume that we obtain some linear measurements of $x_o$ but the number of measurements is much less than the dimension of $x_o$. Suppose that the measure procedure and data are modeled by a matrix $A\in \RR^{m\times n}$ and a vector $b\in\RR^m$. Can we recover $x_o$ from the measure data $b$ such that $Ax_o\simeq b$?

The question above lies in the research field of linear inverse problems that subsume a great number of important applications such as signal denoising/deblurring and compressed sensing. In the case of $m\ll n$, the reconstruction of a general vector $x_0$ from $b$ is impossible even though the data $b$ is not contaminated since the linear system $Ax=b$ is underdetermined. Fortunately, the vector $x_o$ in many assignments has certain structure such as sparsity and low-rankness in compressed sensing. If we know the prior structure of $x_o$ beforehand, then regularization techniques can help us find the desired $x_o$ from the solution set of $Ax=b$. For example, if we previously know that many entries of $x_o$ are zero (in other words, $x_o$ is sparse), then one can use the $\ell_1$ norm as a regularizer to approach the sparse $x_o$ by solving
\begin{equation}\label{BP}
\Min_x\|x\|_1:=\sum_{i=1}^n|x_i|, ~\st~~ Ax=b,
\end{equation}
which is the well-known basis pursuit problem in the field of signal processing \cite{2001Atomic}. More generally, one may need to consider the optimization problem in the following form
\begin{equation}\label{OP}
\Min_x\{E(x)+  \mu  R(x)\},
\end{equation}
where $E(x):=E(Ax,b)$ is chosen to quantify the ``error" between $Ax$ and $b$, $R(x)$ is an appropriate regularizer used to reflect the desired feature of $x_o$, and $\mu>0$ is a penalty parameter playing the role of controlling the trade-off between the data fidelity and the degree of regularization. In order to obtain further improvement in variational image restoration, the Bregman iteration method related to \eqref{OP}, introduced in \cite{2005An}, consists of iteratively solving the following problems
\begin{equation}\label{BI}
x^{k+1}:=\arg\min_x \{E(x)+ \mu D_R^{p^k}(x,x^k)\},
\end{equation}
where $p^k$ is a subgradient of $R$ at $x^k$, i.e., $p^k\in \partial R(x^k)$, and $D_R^{p^k}(x,x^k)=R(x)-R(x^k)-\langle p^k, x-x^k\rangle$ is the Bregman distance of $R$ between $x$ and $x^k$. Compared with \eqref{OP}, the generalized Bregman distance of $R$ between $x$ and $x^k$ replaces the term $R(x)$ to play the role of regularization. If we start with an initial value $x^0$ such that $p^0=0\in \partial R(x^0)$, then the first step of the Bregman iteration method \eqref{BI} is exactly \eqref{OP} since $D_R^{p^0}(x,x^0)=R(x)-R(x^0)$. The second and consecutive steps of the Bregman iteration method aim to generate better and better reconstructions. There are a huge amount of research results, both experimentally and analytically, confirming the superior properties of the Bregman iterations; see for example \cite{2005An,2006Nonlinear,2007Iterative,X2010Bregmanized,2013Higher,2014Color}.

\subsection{Linearized Bregman iterations}
The linearized Bregman iterations (LBreI) method, suggested by Darbon and Osher (2007) and formally introduced in the influential work \cite{yin2008bregman}, replaces the data fidelity term $E(x)$ by the sum of its linearized approximation at $x^k$ and a Euclidean norm proximity term $\frac{1}{2\delta}\|x-x^k\|^2$ in \eqref{BI}. Mathematically, it iteratively solves the following subproblems
\begin{equation}\label{LBI}
x^{k+1}:=\arg\min_x \{\langle \nabla E(x^k), x-x^k\rangle+ \frac{1}{2\delta}\|x-x^k\|^2 + \mu D_R^{p^k}(x,x^k)\},
\end{equation}
After a rearrangement of the terms, the LBreI can be equivalently written in the following form
\begin{equation}\label{LBI1}
x^{k+1}:=\arg\min_x \{ R(x)+ \frac{1}{2\delta\mu}\|x-[x^k-( \delta \nabla E(x^k)-\delta\mu p^k)]\|^2\},
\end{equation}
which can be solved easily in many cases or even in a closed-form such as in the case of $R(\cdot)=\|\cdot\|_1$.
Preliminary convergence results of the LBreI was given in \cite{yin2008bregman} by assuming that $R\in C^2$ is strictly convex over a compact set $\Omega\supset \{x^k\}$. Further convergence properties on the LBreI was deduced in \cite{2009Linearized} under the assumption that the convex function $R(\cdot)$ is continuously differentiable and satisfies some Lipschitz property. However, the simple case of $R(\cdot)=\|\cdot\|_1$, which is a central tool in compressed sensing, fails to satisfy the required assumption. Fortunately, this assumption can be dropped, as shown in the following elegant convergence result.
\begin{theorem}[\cite{2009CONVERGENCE}]\label{Th1}
 Assume that the matrix $A$ is full-rank and $0<\delta <\frac{1}{\|AA^T\|}$. Suppose that $R(x)$ is convex and $\partial R(x)$ is bounded. Then for any fixed $\mu>0$, the sequence $\{x^k\}$ generated by \eqref{LBI1} with $E(x)=\frac{1}{2}\|Ax-b\|^2$ converges to the unique solution of the following problem:
 \begin{equation}\label{AOP}
\Min \{\mu R(x)+\frac{1}{2\delta}\|x\|^2: Ax=b\}.
\end{equation}
\end{theorem}
As the parameter $\mu$ tends to infinity, the term $\frac{1}{2\delta}\|x\|^2$ becomes more and more trivial and finally loses its effect. Actually, this point was clarified in \cite{2009CONVERGENCE} by showing that the unique solution of \eqref{AOP}, denoted by $x_\mu$, tends to a solution of minimizing $R(x)$ subject to $Ax=b$ as $\mu\rightarrow+\infty$; a refiner analysis was done in \cite{Yin2010Analysis} where the parameter $\mu$ only requires to be greater than a certain value. By phrasing \eqref{AOP} as a split feasibility problem and using the concept of Bregman projections, a new deduction of the LBreI and its several new variants were provided in \cite{2014The}.
In order to show global linear convergence of the LBreI, the restricted strongly convex inequality was introduced in \cite{Lai2013Aug}, which was further developed in  \cite{2013Gradient} for accelerated LBreI variants.  The first unified framework that unifies the LBreI and its matrix variant--the singular value thresholding algorithm \cite{cai2010a} was proposed in \cite{2015A} by introducing augmented convex signal recovery models and gauge regularized functions. In order to deal with image deblurring in tight frame domains, the LBreI was modified in \cite{2009Linearizedb} to solve the following problem
 \begin{equation}\label{MAOP}
\Min_x \{\mu R(x)+\frac{1}{2\delta}\|x\|^2: x\in {\arg\min}_z \frac{1}{2}\|Az-b\|^2 \}.
\end{equation}
It should be pointed out that all these mentioned works only consider the special fidelity term $E(x)=\frac{1}{2}\|Ax-b\|^2$. Very recently, some nonconvex extension of the LBreI, allowing $E(x)=E(Ax,b)$ to be in a general form which has a Lipschitz continuous gradient, was made in \cite{benning2021choose}.
Although a group of numerical tests were reported in \cite{benning2021choose} to demonstrate that the LBreI in nonconvex optimization still leads to superior performance than that of the regularized problems \eqref{OP}, the current theory is far from satisfying. On one hand, as partially mentioned in section 4.2 in \cite{benning2021choose}, the required gradient Lipschitz continuity assumption precludes the application of LBreI to many practical problems such as blind deconvolution problems, Poisson inverse problems, and quadratic inverse problems. On the other hand, it is unclear whether similar results to Theorem \ref{Th1} can be established for general convex energy function $E$. These two aspects contribute the main motivation of this study.

\subsection{Beyond gradient Lipschitz continuity}
In order to relax the gradient Lipschitz continuity, we follow a new notion which was recently introduced under the name of Lipschitz-like convexity condition in \cite{Bauschke2016A,Teboull2018A} (also independently rediscovered with the name of relative smoothness in \cite{Lu2016relatively}). If the fidelity term $E(x)$ is two-differentiable, then the gradient Lipschitz continuity is essentially equivalent to that the norm of the Hessian of $E(x)$ can be bounded by a constant.  Let us take a quartic function $E(x)=\frac{1}{12}x^4+\frac{1}{2}x^2$ as an example; as the Hessian $x^2+1$ ``grows" to infinity as $|x|$ tends to infinity, it fails to be gradient Lipschitz continuous. The main idea of the Lipschitz-like convexity condition is to choose a suitable relative function $h(x)$ such that the Hessian of $E(x)$ could be bounded by that of $h(x)$ in the sense of  $L\nabla^2 h(x)-\nabla^2E(x)\succeq 0$ for some constant $L>0$; please refer to \cite{Bauschke2016A} for other equivalent conditions. As the first contribution, we replace the norm proximity term $\frac{1}{2\delta}\|x-x^k\|^2$ in \eqref{LBI} with the Bregman distance term $D_h(x,x^k)$ by choosing a suitable relative function $h(x)$ so that the gradient Lipschitz continuity can be relaxed. Consequently, the generalized method consists in solving the following optimization problems
\begin{equation}\label{MLBI}
x^{k+1}:=\arg\min_x \{\langle \nabla E(x^k), x-x^k\rangle+ \frac{1}{\delta}D_h(x,x^k) + \mu D_R^{p^k}(x,x^k)\}.
\end{equation}
If $\mu=0$, then the Bregman regularization term $\mu D_R^{p^k}(x,x^k)$ disappears, in which case the generalized method \eqref{MLBI} overlaps with the NoLips algorithm and the Bregman proximal gradient method in \cite{Bauschke2016A,bolte2017first}. In other words, our proposed method can also be obtained by adding the term $\mu D_R^{p^k}(x,x^k)$ to the NoLips algorithm or the Bregman proximal gradient method. This point will be highlighted in Section \ref{se3}, where a formal unified framework will be constructed.


\subsection{Contribution and organization}
We summarize the contributions as follows:
\begin{itemize}
  \item For a convex energy function $E$, if it satisfies the Lipschitz-like convexity condition, then we are able to extend Theorem \ref{Th1}. This discovery will correspond to Theorem \ref{mainthm1}.
  \item For a nonconvex energy function $E$, if it satisfies the Lipschitz-like convexity condition and has some ``nice" local properties, then we are able to show the global convergence of generalized LBreI methods. This discovery will correspond to Theorem \ref{mainthm2}.
\end{itemize}

The remainder of the paper is organized as follows. The Bregman distances and the Lipschitz-like convexity condition, as two basic tools, are recalled  in Section \ref{se2}. The unified framework is introduced in Section \ref{se3}. The convergence analysis is presented in Section \ref{se4}. Two application examples are discussed in Section \ref{se5}. Concluding remarks and research directions for future work are given in Section \ref{se6}.

\section{Preliminaries}\label{se2}

 Throughout the paper, we assume that $\EE$ is a finite dimensional vector space with inner product $\langle \cdot, \cdot\rangle$ and induced norm $\|\cdot\|$. For a nonempty subset $\Omega\subseteq \EE$ and a point $x\in \EE$, we define by $\dist(x,\Omega):=\inf_{y\in \Omega}\|x-y\|$ the distance function from $x$ to $\Omega$.  The notation and almost all the facts about convex analysis we employ are standard and can be found in Rockafellar's classic book \cite{Rock1970convex}. Let $f$ be a convex function on $\EE$, the domain (conjugate function of $f$, gradient of $f$, and subgradient of $f$) is denoted by $\dom{f}$ ($f^*$, $\nabla f$, $\partial f$, respectively).

 \subsection{The Bregman distances}
 The most important technical ingredient in the linearized Bregman iterations (also in our proposed algorithmic framework) is the concept of Bregman distance, which was originally introduced by Bregman in the 1967's paper \cite{1967The} for extending the classical method of cyclic orthogonal projections to the case of non-orthogonal projections.  Given a differentiable convex function $h$, the Bregman distance associated with $h$ between two points $x$ and $y$ is defined as
 $$D_h(x,y):=h(x)-h(y)-\langle \nabla h(y), x-y\rangle.$$
It is worth noting that $D_h$ is not a metric since the symmetry and the triangle inequality fail to hold.
In order to guarantee that projection methods equipped with the Bregman distance still behave well,  a ``nice" convex function $h$ has to be chosen.
The class of Legendre functions, which was investigated in \cite{1997Legendre}, have become a popular choice to define the Bregman distance \cite{Bauschke2016A}. Now, we recall its definition below.

\begin{definition}[Legendre functions, \cite{Rock1970convex}]
Let $h: \EE \rightarrow(-\infty, +\infty]$ be a proper lower semicontinuous (lsc) convex function. We say that
 \begin{itemize}
   \item  it is essentially smooth if $\rint\dom{h}\neq\emptyset$, $h$ is differentiable on $\rint\dom{h}$, and  $\|\nabla h(x^k)\|\rightarrow \infty$ for every sequence $\{x^k\}_{k\geq 0}\subseteq \rint\dom{h}$ converging to a boundary point of $\dom{h}$ as $k\rightarrow \infty$,
   \item and it is of Legendre type if $h$ is essentially smooth and strictly convex on $\rint\dom{h}$.
 \end{itemize}
\end{definition}

Note that the Legendre functions are essentially smooth and hence differentiable on $\rint\dom{h}$. However, the associated function $R$ in the linearized Bregman iterations may be non-differentiable on $\rint\dom{R}$; for example $R(\cdot)=\|\cdot\|_1$ is not of Legendre type. Therefore, we need the concept of generalized Bregman distances, introduced by Kiwiel in \cite{1997proximal}. It is worth emphasizing that we do not completely follow the definition of Kiwiel in \cite{1997proximal} where the associated function is needed to be strictly convex on its domain, excluding the case of $R(\cdot)=\|\cdot\|_1$ again. Given a proper lsc convex function $h$, the generalized Bregman distance associated with $h$ between $x, y\in \EE$ with respect to a subgradient $y^*\in\partial h(y)$ is defined by
\begin{equation}\label{Bregd}
D_h^{y^*}(x,y):=h(x)-h(y)-\langle y^*, x-y\rangle, \forall x\in \dom{h},~ y\in \dom{\partial h}.
\end{equation}
Using this generalized definition, we can obtain important lemma which generalizes the three points identity in \cite{Chen1993conv}.
\begin{lemma}[\cite{1997Free,1997proximal}]\label{lemBreg}
Let $h:\EE\rightarrow(-\infty, +\infty]$ be a proper lsc convex function with $\dom\partial h\neq \emptyset$. For any $z\in \dom{h}$ and $ x,y\in \dom{\partial h}$ and $x^*\in \partial h (x), y^*\in \partial h (y)$, we have that
\begin{equation}\label{Bregdis1}
 D_h^{x^*}(z,x) - D_h^{y^*}(z,y) - D_h^{x^*}(y,x)=\langle x^*-y^*, y-z\rangle.
\end{equation}
\end{lemma}
Recall the fact that for a Legendre function $h$, it holds that  $\dom \partial h=\rint\dom{h}$ with $\partial h(x)=\{\nabla h(x)\}$ for any $x\in\rint\dom{h}$. Therefor, applying Lemma \ref{lemBreg}, we recover the well-known three points identity
\begin{equation}\label{threep}
D_h(z,x) - D_h(z,y) - D_h(y,x)=\langle\nabla h(x)-\nabla h(y), y-z\rangle,
\end{equation}
which holds for any Legendre function $h$ with the three points $z\in \dom{h}$ and $x,y\in  \rint\dom{h}$.

At last, we introduce a measure for the lack of symmetry in $D_h$.
\begin{definition}[\cite{Bauschke2016A}]\label{def1}
Given a Legendre function $h:\EE\rightarrow(-\infty, +\infty]$, its symmetry coefficient is defined by
$$\alpha(h):=\inf \left\{\frac{D_h(x,y)}{D_h(y,x)}: x, y\in \rint\dom{h}, x\neq y\right\}\in [0,1].$$
\end{definition}

\subsection{The Lipschitz-like convexity condition}
The applied scope of the well-known proximal gradient method and its variants is limited by the gradient Lipschitz continuity assumption. Recently, the authors of \cite{Bauschke2016A} proposed the Lipschitz-like/convexity condition as an alternative to go beyond the gradient Lipschitz continuity. Below, we recall its definition.

\begin{definition}[Lipschitz-like/convexity condition, \cite{Bauschke2016A}\cite{Teboull2018A}]
Let $h: \EE \rightarrow(-\infty, +\infty]$ be a Legendre function and let $E: \EE \rightarrow(-\infty, +\infty]$ be a proper lsc function with $\dom E\supset \dom h$, and $E$ is differentiable on $\rint\dom{h}$. Given such a pair of functions $(E,h)$, the Lipschitz-like/convexity condition denoted by (\textbf{LC}) is:
\begin{center}
 (\textbf{LC})~~~$\exists ~L>0$~~with $Lh-E$ convex on $\rint\dom{h}$.
\end{center}
\end{definition}

The following three points extended descent lemma will play a very important role in the forthcoming analysis.
\begin{lemma}\label{lemDesc}
Consider the pair of functions $(E,h)$ as above and take $L>0$. Then, the function $Lh-E$ is convex on $\rint\dom h$ if and only if for any $(x,y,z)\in (\rint\dom h)^3$:
\begin{equation}\label{Desc}
 E(x)\leq E(y)+\langle \nabla E(z), x-y\rangle + LD_h(x,z)-D_E(y,z).
\end{equation}
In particular,  assume that (\textbf{LC}) holds for the pair of convex functions $(E,h)$. Then, for any $(x,y,z)\in \rint\dom{h}\times\dom{h}\times\rint\dom{h}$, we have
\begin{equation}\label{Desc1}
 E(x)\leq E(y)+\langle \nabla E(z), x-y\rangle + LD_h(x,z).
\end{equation}
\end{lemma}
It is not hard to see that \eqref{Desc1} follows from \eqref{Desc} by noting that $D_E(y,z)\geq 0$ when $E$ is convex.
For a possibly nonconvex function $E$, the term $D_E(y,z)$ is difficult to bound. Therefore, we usually  consider the case of $y=z$, in which the term $D_E(y,z)$ disappears and the corresponding descent inequality becomes:
\begin{equation}\label{Desc2}
 E(x)\leq E(y)+\langle \nabla E(y), x-y\rangle + LD_h(x,y).
\end{equation}
When $h(\cdot)=\frac{1}{2}\|\cdot\|^2$, it further reduces to
\begin{equation}\label{Desc3}
 E(x)\leq E(y)+\langle \nabla E(y), x-y\rangle + \frac{L}{2}\|x-y\|^2,
\end{equation}
which is exactly the classical descent lemma; see for example Lemma 5.7 in \cite{Beck2017First}.

\section{The proposed algorithmic framework}\label{se3}
Let $x_0\in\EE, p^0\in \partial R(x_0)$ be given. We propose the following linearized Bregman iteration framework, abbreviated as LBreIF, to extend the LBreI method.
\begin{subequations}
\begin{align}[left = \empheqlbrace\,]
\label{a1}& x^{k+1}=\arg\min_x \{\langle \nabla E(x^k), x-x^k\rangle+ \beta R(x)+\frac{1}{\delta^k}D_h(x,x^k) + \mu^k D_R^{p^k}(x,x^k)\},   \\
\label{a2}&  p^{k+1}= \frac{\mu^k}{\mu^k+\beta~}p^k-\frac{1}{\delta^k(\mu^k+\beta~)}\left[\nabla h(x^{k+1})-\nabla h(x^k)+\delta^k\nabla E(x^k)\right].
\end{align}
\end{subequations}
To help the reader understand LBreIF better, we add a few remarks:
\begin{itemize}
  \item The relative function $h$ is chosen to be strongly convex on its domain so that the subproblem in \eqref{a1} has the unique solution $x^{k+1}$. The term $D_h(x,x^k)$ is used to replace/generalize the  proximity term $\frac{1}{2}\|x-x^k\|^2$ in the LBreI \eqref{LBI} so that the traditional gradient Lipschitz continuity assumption can be relaxed. Meanwhile, a potential difficulty is how to determine such a desired function $h$ so that the pair of $(E,h)$ satisfies the Lipschitz-like convexity condition and at the same time the subproblem \eqref{a1} can be solved efficiently.
  \item  The term $\beta R(x)$ is taken into account in \eqref{a1} so the NoLips algorithm in \cite{Bauschke2016A} can be included as a special case of \eqref{a1} by setting $\mu^k\equiv0$. In this way, we can see that the main difference between NoLips and LBreIF is the term $D_R^{p^k}(x,x^k)$, which helps us generate regularized solutions via the \textsl{nonsmooth} Bregman distance $D_R(\cdot,\cdot)$; whereas NoLips as well as the traditional regularization algorithms produce solutions via optimizing the sum of a fidelity function and a regularized function. Theoretically, the sequence $\{x^k\}$ generated by LBreIF enjoys a refiner convergence phenomenon, as shown in Theorem \ref{mainthm1}.
  \item The subgradient $p^k\in \partial R(x^k)$ in \eqref{a2} is updated iteratively according to the optimality condition of \eqref{a1}. The reader will find that \eqref{a2} plays a key role in the forthcoming convergence analysis.
\end{itemize}
Now, we present an elegant expression for LBreIF in the special case of $\beta\equiv 0$, $\mu_k\equiv\mu$, and $\delta^k\equiv \delta$. To do this, we first reformulate \eqref{a2} as follows
\begin{equation}\label{a22}
 \nabla h(x^{k+1})+\delta\mu p^{k+1}=\nabla h(x^{k})+\delta\mu p^{k} -\delta \nabla E(x^k).
\end{equation}
Denote $y^k:=\nabla h(x^{k})+\delta\mu p^{k}$; then we have
\begin{subequations}
\begin{align}[left = \empheqlbrace\,]
\label{m1}& y^{k+1}= y^k  -\delta \nabla E(x^k),   \\
\label{m2}&  x^{k+1}= (\nabla h +\delta\mu\partial R)^{-1}(y^{k+1}).
\end{align}
\end{subequations}
In particular, the iterative scheme above with $R(x)\equiv 0$ returns to the well-known mirror descent
$$x^{k+1}=(\nabla h)^{-1}(\nabla h(x^k)  -\delta \nabla E(x^k)),$$
where $\nabla h$ is the mirror mapping and $(\nabla h)^{-1}$ is the inverse mirror mapping \cite{2004Prox}. Thereby, the iterative scheme \eqref{m1}-\eqref{m2} may be viewed as a generalized mirror descent method if we view $\nabla h+\delta\mu \partial R$ as the mirror mapping and $(\nabla h +\delta\mu\partial R)^{-1}$ as the inverse mirror mapping \cite{2021Mirror}.

%

At last, we introduce a natural assumption on the regularized function $R$ and the relative function $h$ so that the iterate \eqref{a1} is well defined if $\mu^k>0$, $\delta^k>0$ and $x^k\in \rint\dom{h}$.
\begin{lemma}\label{wellposed}
If $h+\lambda R$ is supercoercive for all $\lambda>0$ in the sense that
\begin{equation}\label{sup}
{\lim\inf}_{\|x\|\rightarrow \infty}\frac{h(x)+\lambda R(x)}{\|x\|}=\infty,
 \end{equation}
then for any $x^k\in \rint\dom{h}$, the subproblem in \eqref{a1} has the unique solution $x^{k+1}$ and it must belong to $\rint\dom{h}$.
\end{lemma}
\begin{proof}
Using the definitions of the Bregman distances, the objective function in \eqref{a1} can be rewritten as
$$g(x):=h(x)+\lambda R(x)+\langle \bar{u}, x\rangle +C,$$
where $\bar{u}\in \EE$ and $C\in \RR$ are constant quantities. Note that $|\langle \bar{u}, x\rangle|\leq \|\bar{u}\|\|x\|$ by the Cauchy-Schwartz inequality. We conclude that the objective function $g$ is still supercoercive. Thus, the set of minimizers of $g$ must be nonempty and compact by Weierstrass' theorem \cite{Beck2017First}. The uniqueness of solution follows from the strict convexity of $h$ on its domain. Finally, by the optimality condition we have
$$0\in \partial h(x^{k+1})+\lambda\partial R(x^{k+1})+ \bar{u},$$
which implies that $x^{k+1}\in \dom\partial h= \rint\dom{h}.$ This completes the proof.
\end{proof}

It can be easily verified that for any strongly convex function $h$ and nonnegative function $R$ (in the sense of $R(x)\geq 0$), the composition function $h+\lambda R$ is always supercoercive for all $\lambda>0$.

\section{Convergence analysis}\label{se4}
In this section, we present a detailed convergence analysis for the proposed algorithmic framework.

\subsection{Assumptions and the key lemma}
The following assumption is a basic setting about the involved functions.
\begin{assumption}\label{ass}
We assume that the relative function $h$, the regularized function $R$, and the loss function $E$ satisfy the following conditions:
\begin{itemize}
  \item[(i)] $h:\EE\rightarrow (-\infty, +\infty]$ is of Legendre type.
  \item[(ii)] $R:\EE\rightarrow [0, +\infty]$ is proper lsc convex with $\dom \partial R\supset \rint\dom h$ and $\dom R\bigcap \rint\dom h\neq \emptyset$.
  \item[(iii)] $E:\EE\rightarrow (-\infty, +\infty]$ is proper lsc with $\dom E\supset \dom h$ and is differentiable on $\rint\dom h$ and continuous on $\dom h$. Moreover, (\textbf{LC}) holds for the pair of functions $(E,h)$.
  \item[(iv)] $-\infty< \inf_{x\in\dom h} f_\beta(x)$ with  $ f_\beta(x):=\beta R(x)+E(x)$ for some $\beta\geq 0$.
\end{itemize}
\end{assumption}
The conditions $\dom \partial R\supset \rint\dom h$ and $\dom R\bigcap \rint\dom h\neq \emptyset$ in item (ii) and the differentiablity of $E$ on $\rint\dom h$ in item (iii) are posed to guarantee the objective function in \eqref{a1} is well-defined if $x^k\in \rint\dom h$ . Item (ii) can be satisfied when $R$ is real-valued since $\dom \partial R= \dom R=\EE$. Hence, many exiting regularized functions like the $\ell_1$ norm obey these restrictions on $R$. Items (i) and (iii) essentially requires the (\textbf{LC}) condition. The last item holds trivially for \textsl{nonnegative} regularized function $R$ and energy function $E$.

In order to derive point convergence results, we will rely on the following assumption:
\begin{assumption}\label{ass1}
The Bregman distance associated with the relative function $h$ and the regularized function $R$ satisfy the following conditions:
\begin{itemize}
  \item[(i)]   For every $x\in \dom h$ and $\eta\in \RR$, one of the level sets $\{y\in\rint\dom h: D_h(x,y)\leq \eta\}$ and $\{y\in \dom R: D_R(x,y)\leq \eta\}$ is bounded.
  \item[(ii)] If $\{x^k\}\subset \rint\dom h$ converges to some $x$ in $\dom h$, then $D_h(x, x^k)\rightarrow 0$ and $D_R(x, x^k)\rightarrow 0$.
  \item[(iii)] If $x\in \rint\dom h$ and if $\{x^k\}$  is such that $D_h(x, x^k)\rightarrow 0$ or $D_R(x, x^k)\rightarrow 0$, then $x^k\rightarrow x$.
\end{itemize}
\end{assumption}
If we drop the requirement $D_R(x, x^k)\rightarrow 0$ in item (ii) above, then Assumption \ref{ass1} can be implied by Assumption H in \cite{Bauschke2016A} and hence it holds for many entropies such as the Boltzman-Shannon function $h(x)=x\log x$ which will appear in the section of application. If $R$ itself is real-valued convex, then we always have that $D_R(x, x^k)\rightarrow 0$ as $x^k\rightarrow x$ due to the continuity of $R$ and the boundedness of subgradients over compact sets; see Theorem 3.16 in \cite{Beck2017First}.

For nonconvex convergence analysis, we require the following assumptions. The first one was used in \cite{benning2021choose} and the second one in \cite{bolte2017first}. Again, if $R$ itself is real-valued convex, then the firs assumption below holds trivially.
\begin{assumption}\label{ass2}
The regularized function $R$ has locally bounded subgradients in the sense that if for every bounded set $U\subset \EE$ there exists a constant $C>0$ such that for every $x\in U$ and for all $p\in \partial R(x)$ we have $\|p\|\leq C$.
\end{assumption}

\begin{assumption}\label{ass3}
A function $f$ is said to be locally gradient-Lipschitz-continuous if for every bounded set $U\subset \rint\dom f$ there exists a constant $L_f>0$ such that for any $x, y\in U$ we have $\|\nabla f(x)-\nabla f(y)\|\leq L_f\|x-y\|$.
\end{assumption}

The next lemma provides key descent estimations for the objective function in terms of the Bregman distances.
\begin{lemma}[Descent inequalities]\label{dein}
Under Assumption \ref{ass}, we have
\begin{eqnarray} \label{desineq0}
\begin{array}{lll}
&&\beta R(x^{k+1})+E(x^{k+1})-(\beta R(x^k)+E(x^k) )\\
&\leq & LD_h(x^{k+1},x^k) -\frac{1}{\delta^k}D_h(x^k,x^{k+1})-\frac{1}{\delta^k}D_h(x^{k+1},x^k)\\
&& -\mu^kD_R(x^k,x^{k+1})-\mu^kD_R(x^{k+1},x^k).
\end{array}
\end{eqnarray}
If $E$ is also assumed convex, then for any $x\in \dom h$, it holds that
\begin{eqnarray} \label{desineq}
\begin{array}{lll}
&&\beta R(x^{k+1})+E(x^{k+1})-(\beta R(x)+E(x) )\\
&\leq &\mu^kD_R(x,x^k)-\mu^kD_R(x,x^{k+1})-\mu^kD_R(x^{k+1},x^k)+\\
&& \frac{1}{\delta^k}D_h(x,x^k)-\frac{1}{\delta^k}D_h(x,x^{k+1})-\frac{1}{\delta^k}D_h(x^{k+1},x^k)+LD_h(x^{k+1},x^k).
\end{array}
\end{eqnarray}
\end{lemma}
\begin{proof}
We only show \eqref{desineq}; the other relationship can be shown in a similar way. To this end, we first rephrase \eqref{a2} as the following equality:
\begin{equation}\label{Basic}
\mu^k(p^{k+1}-p^k)+\beta p^{k+1}+\nabla E(x^k)+\frac{1}{\delta^k}(\nabla h(x^{k+1})-\nabla h(x^k))=0.
\end{equation}
Now, let us do the inner product between the left-hand side of \eqref{Basic} and the term $x-x^{k+1}$. According to the generalized three points identity in Lemma \ref{lemBreg}, we have
\begin{equation}\label{b1}
\mu^k\langle p^{k+1}-p^k, x-x^{k+1}\rangle =\mu^kD_R(x,x^k)-\mu^kD_R(x,x^{k+1})-\mu^kD_R(x^{k+1},x^k).
\end{equation}
Similarly, applying the three points identity to the relative function $h$, we obtain
\begin{equation}\label{b2}
\frac{1}{\delta^k} \langle \nabla h(x^{k+1})-\nabla h(x^k), x-x^{k+1}\rangle =\frac{1}{\delta^k}D_h(x,x^k)-\frac{1}{\delta^k}D_h(x,x^{k+1})-\frac{1}{\delta^k}D_h(x^{k+1},x^k).
\end{equation}
Using the three points extended descent property in Lemma \ref{lemDesc}, we have
\begin{equation}\label{b3}
\langle \nabla E(x^k), x-x^{k+1}\rangle \leq E(x)-E(x^{k+1})+LD_h(x^{k+1},x^k).
\end{equation}
Using the subgradient inequality for the convex function $R$, we have
\begin{equation}\label{b4}
\beta \langle p^{k+1}, x-x^{k+1}\rangle \leq \beta  R(x)-\beta  R(x^{k+1}).
\end{equation}
Thereby, the descent inequality \eqref{desineq} follows from \eqref{Basic}-\eqref{b4}.
\end{proof}

\subsection{Convergence for convex optimization}
The basic convergence properties of LBreIF in the convex case are presented in the following lemma.
\begin{lemma}[Basic convergence]\label{mainlemma}
Let $\sigma_k:=\sum_{i=0}^k\delta^k$, $\upsilon:=\inf_{x\in\dom h}f_\beta(x)$, and  let $\{x^k\}$ be the sequence generated by \eqref{a1}-\eqref{a2} with the parameters satisfying
\begin{subequations}
\begin{align}[left = \empheqlbrace\,]
\label{p1}& 0<\delta^k< \frac{1+\alpha(h)-\omega}{L}, \exists \omega\in (0,1+\alpha(h)),\\
\label{p2}& 0<\tau< \delta^{k+1}\mu^{k+1}\leq \delta^k\mu^k, \forall k\geq 0.
\end{align}
\end{subequations}
If Assumption \ref{ass} holds and $E$ is also assumed convex, then we have that
\begin{enumerate}
  \item[(a)] (Monotonicity). $\{f_\beta(x^k)\}$ is nonincreasing.
  \item[(b)] (Summability).  $\sum_{k=0}^\infty D_h(x^{k+1},x^k)<\infty.$
  \item[(c)] (Convergence of the function values). If $\sigma_k\rightarrow \infty$, then $\lim_{k\rightarrow \infty}f_\beta(x^k)=\upsilon.$
\end{enumerate}
\end{lemma}
\begin{proof}
First of all, denote
$H(x,x^k):=\delta^k\mu^kD_R(x,x^k)+D_h(x,x^k).$
Using \eqref{desineq} of Lemma \ref{dein} and the condition \eqref{p2} on $\delta^k$ and $\mu^k$, we obtain that for all $x\in \dom h$,
\begin{equation}\label{maineq}
\delta^k(f_\beta(x^{k+1})-f_\beta(x))\leq H(x,x^k)-H(x,x^{k+1})-(1-\delta^k L)D_h(x^{k+1},x^k).
\end{equation}
Setting $x=x^k$ in \eqref{maineq}, recalling the definition of $H(x,x^k)$, and using the condition \eqref{p1} on $\delta^k$, we derive that
\begin{eqnarray} \label{meq1}
\begin{array}{lll}
\delta^k(f_\beta(x^{k+1})-f_\beta(x^k))&\leq & -H(x^k,x^{k+1})-(1-\delta^k L)D_h(x^{k+1},x^k)\\
&\leq & -D_h(x^k,x^{k+1})-(1-\delta^k L)D_h(x^{k+1},x^k) \\
&\leq & -D_h(x^k,x^{k+1})+\alpha(h)D_h(x^{k+1},x^k)-\omega D_h(x^{k+1},x^k)\\
&\leq & -\omega D_h(x^{k+1},x^k),
\end{array}
\end{eqnarray}
 where the last relationship follows from the definition of $\alpha(h)$. Therefore, the sequence $\{f_\beta(x^k)\}$ is nonincreasing. Note that $x^k\in \dom h$ and the assumption that $\upsilon=\inf_{x\in\dom h}f_\beta(x)>-\infty$, we conclude that
\begin{equation}\label{meq2}
\lim_{k\rightarrow \infty}f_\beta(x^k)\geq \upsilon>-\infty.
\end{equation}
Using \eqref{meq1}, we get
$$D_h(x^{k+1},x^k)\leq \frac{\delta^k}{\omega}(f_\beta(x^k)-f_\beta(x^{k+1}))\leq \frac{1+\alpha(h)-\omega}{\omega L}(f_\beta(x^k)-f_\beta(x^{k+1})).$$
Thus, it follows from \eqref{meq2} and the inequality above that for any $n\in \mathbb{N}$,
$$\sum_{k=0}^n D_h(x^{k+1},x^k)\leq \frac{1+\alpha(h)-\omega}{\omega L}(f_\beta(x^0)-f_\beta(x^{n+1}))\leq \frac{1+\alpha(h)-\omega}{\omega L}(f_\beta(x^0)-\upsilon).$$
This implies the following summability
\begin{equation}\label{meq3}
\sum_{k=0}^\infty D_h(x^{k+1},x^k)\leq \frac{1+\alpha(h)-\omega}{\omega L}(f_\beta(x^0)-\upsilon)<\infty.
\end{equation}
Now, let us show the convergence of the function values. Summing \eqref{maineq} over $k=0,\cdots, n$,  using the fact $\delta^kL-1\leq \alpha(h)$, and noting the nonnegativity of $H(x,x^k)$, we derive that
\begin{eqnarray} \label{meq4}
\begin{array}{lll}
\sum_{k=0}^n \delta^k(f_\beta(x^{k+1})-f_\beta(x))&\leq &H(x,x^0)-H(x,x^{n+1})+\alpha(h)\sum_{k=0}^n  D_h(x^{k+1},x^k)\\
&\leq &H(x,x^0) +\alpha(h)\sum_{k=0}^\infty D_h(x^{k+1},x^k).
\end{array}
\end{eqnarray}
Since the sequence $\{f_\beta(x^k)\}$ is nonincreasing, we get  $$\sum_{k=0}^n \delta^k(f_\beta(x^{k+1})-f_\beta(x))\geq (\sum_{k=0}^n \delta^k)(f_\beta(x^{n+1})-f_\beta(x)).$$
Therefore, invoking \eqref{meq4} and the notation $\sum_{k=0}^n \delta^k=\sigma_n$ we further get
$$f_\beta(x^{n+1})-f_\beta(x)\leq \frac{H(x,x^0) +\alpha(h)\sum_{k=0}^\infty D_h(x^{k+1},x^k)}{\sigma_n}.$$
Taking the limit with $\sigma_n\rightarrow \infty$ above and recalling that  $\sum_{k=0}^\infty D_h(x^{k+1},x^k)<\infty$, we get that for every $x\in \dom h$, $\lim_{n\rightarrow \infty}f_\beta(x^{n+1})\leq f_\beta(x)$ and hence $
 \lim_{k\rightarrow \infty}f_\beta(x^k)\leq \upsilon.$
Together with \eqref{meq2}, it follows that $\lim_{k\rightarrow \infty}f_\beta(x^k)=\upsilon.$ This completes the proof.
\end{proof}

In order to obtain point convergence results, we need the following basic result about sequence convergence.
\begin{lemma}[\cite{Polyak1987,Bauschke2016A}]\label{polyaklemma}
Let $\{v_k\}$ and $\{\epsilon_k\}$ be nonnegative sequences. Assume that $\sum_{k=0}^\infty \epsilon_k<\infty$ and that $$v_{k+1}\leq v_k +\epsilon_k.$$
Then, $\{v_k\}$ converges.
\end{lemma}

We now present the point convergence results.
\begin{lemma}[Point convergence]\label{mainlemma2}
Let  Assumption \ref{ass} hold, $E$ be convex, and $\beta\geq 0$. Denote
$$\cX_\beta:=\arg\min\{f_\beta(x): x\in \overline{\dom h}\}.$$
\begin{itemize}
  \item[(1)] If $\cX_\beta$ is nonempty, then any limit point of $\{x^k\}$ (if it exists) belongs to $\cX_\beta$.
  \item[(2)] If  $\overline{\dom h}=\dom h$, Assumption \ref{ass1} holds, and $\cX_\beta$ is nonempty, then the sequence of  $\{x^k\}$ converges to some solution $x^*\in \cX_\beta.$
\end{itemize}
\end{lemma}

\begin{proof}
Let $x^*$ be a limit point of  $\{x^k\}$. In other words, there exists a subsequence $\{x^{k_i}\}$ such that $\lim_i x^{k_i}=x^*$. Then, $x^*\in \overline{\dom h}$ due to the closedness of $\overline{\dom h}$ and $\{x^{k_i}\}\subset \overline{\dom h}$. By using (a) and (c) of Lemma \ref{mainlemma} and the lower semicontinuity of $E$ and $R$, we derive that
\begin{eqnarray}
\begin{array}{lll}
\min\{f_\beta(x): x\in \overline{\dom h}\} &= & \min\{f_\beta(x): x\in \dom h\}\\
&=& \lim_{k\rightarrow\infty} f_\beta(x^k)= \lim_{i\rightarrow\infty} f_\beta(x^{k_i})\\
&= &  \lim\inf_{i\rightarrow\infty} \{E(x^{k_i}) +\beta R(x^{k_i})\}\\
&\geq & \lim\inf_{i\rightarrow\infty} E(x^{k_i}) + \lim\inf_{k\rightarrow\infty}\beta R(x^{k_i})\\
&\geq & E(x^*)+\beta R(x^*)=f_\beta(x^*),
\end{array}
\end{eqnarray}
which implies that $x^*\in \cX_\beta$. This completes the proof of (1).

Recalling \eqref{maineq}, we have for all $x\in \dom h=\overline{\dom h}$,
\begin{equation}
\delta^k(f_\beta(x^{k+1})-f_\beta(x))\leq H(x,x^k)-H(x,x^{k+1})-(1-\delta^k L)D_h(x^{k+1},x^k).
\end{equation}
Fix $\bar{x}\in \cX_\beta$. Due to $\delta^k(f_\beta(x^{k+1})-f_\beta(\bar{x}))\geq 0$ and the condition \eqref{p1} on $\delta^k$, the inequality above yields
\begin{equation}
H(\bar{x},x^{k+1})\leq H(\bar{x},x^k)+(\alpha(h)-\omega)D_h(x^{k+1},x^k).
\end{equation}

If $\alpha(h)-\omega\leq 0$, then the nonnegative sequence $\{H(\bar{x},x^k)\}$ is nonincreasing and hence it must converge. If $\alpha(h)-\omega > 0$, then we can use Lemma \ref{polyaklemma} with $\epsilon_k=(\alpha(h)-\omega)D_h(x^{k+1},x^k)$ to conclude that $\{H(\bar{x},x^k)\}$ converge since $\sum_{k=0}^\infty D_h(x^{k+1},x^k)<\infty$ from (b) of Lemma \ref{mainlemma}. The convergence of $\{H(\bar{x},x^k)\}$ and condition (i) of Assumption \ref{ass1} imply that the sequence $\{x^k\}$ is bounded. Hence, we can take one of its cluster points, denoted by $x^*$, which must lies in $\dom h=\overline{\dom h}$. Further, we deduce that $x^*\in \cX_\beta$ from part (1).

Let $\{x^{n_k}\}$ be the subsequence of $\{x^k\}$ such that $x^{n_k}\rightarrow x^*$. By condition (ii) of Assumption \ref{ass1}, we have that $D_h(x^*,x^{n_k})\rightarrow 0$ and $D_R(x^*, x^{n_k})\rightarrow 0$. Thus, $H(x^*,x^{n_k})=\delta^{n_k}\mu^{n_k}D_R(x^*, x^{n_k})+D_h(x^*,x^{n_k})\rightarrow 0$ as well. Note that $\{H(x^*,x^k)\}$ is a convergence sequence. Hence it must converge to zero. Recall that $\delta^k\mu^k$ are bounded by $\tau$. We can deduce that both $\{D_R(x^*, x^{k})\}$ and $\{D_h(x^*,x^{k})\}$  converge to zero, which implies that $\{x^k\}$ converges to $x^*$ by condition (iii) of Assumption \ref{ass1}.
\end{proof}


Built on the lemmas above, we are now to present the first main convergence theorem in this study.
\begin{theorem}\label{mainthm1}
Suppose that Assumptions \ref{ass}-\ref{ass2} hold, the minimizer set $\cX_{0}$ is nonempty, $\overline{\dom h}=\dom h$, and the function $E(x)$ has the form of $E(Ax,b)$,
satisfying that $A$ is surjective and
\begin{equation}\label{qg}
\psi(E(u,b)-E(\hat{u},b))\geq \|u-\hat{u}\|,~\forall u\in \dom E(\cdot, b),
\end{equation}
where $\psi(t)$ is some real function with $\lim_{t\rightarrow 0} \psi(t)=0$ and $\hat{u}\in {\arg\min}_w\{E(w,b)\}$.
Suppose also that $h$ is continuously differentiable on $\rint\dom h$. Denote
$\hat{x}:= \arg\min\{\mu R(x)+\frac{1}{\delta}h(x), x\in \cX_0\}.$
Then, the sequence $\{x^k\}$, generated by \eqref{a1}-\eqref{a2} with $\mu^k \equiv \mu$, $\delta^k\equiv \delta$, $\beta=0$, and $x^0=p^0=0$,  converges either to the unique minimizer $\hat{x}$ or to a boundary point of $\dom h$.
\end{theorem}
\begin{proof}
From Lemma \ref{mainlemma2}, we know that the sequence $\{x^k\}$ must converge to some point $x^*$. If this point is not a boundary point of $\dom h$, let us show $x^*=\hat{x}$.
We begin with the relationship \eqref{a2} in the current setting, which has the following form
\begin{equation}\label{pq}
\mu p^{k+1}+\frac{1}{\delta}\nabla h(x^{k+1})= \mu p^{k}+\frac{1}{\delta}\nabla h(x^{k})-\nabla E(x^k).
\end{equation}
Denote $q^k:=\mu p^{k}+\frac{1}{\delta}\nabla h(x^{k})$ and let $g(x)=\mu R(x)+\frac{1}{\delta}h(x)$, which must be a proper lsc convex function; then $q^k\in \partial g(x^k)$.
Using the subgradient inequality, we derive that
\begin{eqnarray}\label{lowb}
\begin{array}{lll}
g(x^k) &\leq  & g(\hat{x})-\langle q^k, \hat{x}-x^k\rangle \\
&= &  g(\hat{x})+\langle \sum_{i=0}^{k-1}\nabla E(Ax^i,b), A\hat{x}-A x^k\rangle \\
&\leq  &  g(\hat{x})+\| \sum_{i=0}^{k-1}\nabla E(Ax^i,b)\| \|A\hat{x}-A x^k\|.
\end{array}
\end{eqnarray}
Denote $z^{k-1}:=\sum_{i=0}^{k-1}\nabla E(Ax^i,b)$. Now, we show two results: (a). $\{z^k\}$ is bounded and (b). $\|A\hat{x}-A x^k\|\rightarrow 0$ as $k\rightarrow \infty$. Actually, from \eqref{pq} we know
$$q^{k+1}=q^k-\nabla E(x^k)=\cdots =q^0-A^T\sum_{i=0}^{k}\nabla E(Ax^i,b)=q^0-A^Tz^k.$$
Since $A$ is a surjective mapping, the verification of (a) can be reduced to showing the boundedness of $\{q^k\}$. Due to the convergence of $\{x^k\}$ and Assumption \ref{ass2}, $\{p^k\}$ must be bounded. On the other hand, since $\{x^k\}$ converges to $x^*\in \rint\dom h$ and $\nabla h(x)$ is continuously differentiable at $x^*$, $\{\nabla h(x^{k})\}$ must converge and hence it is a bounded sequence. Recalling $q^k=\mu p^{k}+\frac{1}{\delta}\nabla h(x^{k})$, we can conclude that $\{q^k\}$ is bounded. It remains to show (b). Since $\hat{x}$ is a interior point of $h$ and that $\hat{x}\in \cX_0=\arg\min\{E(x): x\in \overline{\dom h}\}$, we have $\hat{x}\in \arg\min\{E(x)\}$ and hence $\nabla E(\hat{x})=A^T\nabla E(A\hat{x},b)=0$, which further implies that $\nabla E(A\hat{x},b)=0$ due to the fact of $A$ being surjective. Therefore, we get $A\hat{x}\in{\arg\min}_{w}\{E(w,b)\}$ and hence the condition \eqref{qg} now can be used to deduce that as $k\rightarrow \infty$,
$$\|A\hat{x}-A x^k\|\leq \psi(E(A\hat{x},b)-E(A x^k,b)=\psi(E(\hat{x})-E(x^k))\rightarrow 0,$$
where the last relationship follows from $\lim_{t\rightarrow 0} \psi(t)=0$  and the fact $\lim_{k\rightarrow \infty} E(x^k)=\lim_{k\rightarrow \infty}E(\hat{x})$, implied by (c) of Lemma \ref{mainlemma}.

Now, using \eqref{lowb} and the verified results (a)-(b), we have
$g(x^*)\leq {\lim\inf}_{k\rightarrow \infty}g(x^k)\leq g(\hat{x}).$
Therefore, $x^*\in \arg\min\{\mu R(x)+\frac{1}{\delta}h(x), x\in \cX_0\},$ which implies that $x^*=\hat{x}$ by the uniqueness of solutions. Note that the uniqueness follows from the strict convexity of $h$. This completes the proof.
\end{proof}

In particular, if $h(x)=\frac{1}{2}\|x\|^2$ whose boundary set is empty, we have the following result which generalizes Theorem 1 about LBreI.
\begin{corollary}
Let $h(x)=\frac{1}{2}\|x\|^2$ and suppose that Assumptions \ref{ass}-\ref{ass2} hold for $E$ and $R$, the minimizer set $\cX_{0}$ is nonempty, and the function $E(x)$ has the form of $E(Ax,b)$,
satisfying that $A$ is surjective and the condition \eqref{qg} holds.
Denote $\hat{x}:= \arg\min\{\mu R(x)+\frac{1}{\delta}h(x), x\in \cX_0\}.$
Then, the sequence $\{x^k\}$, generated by \eqref{a1}-\eqref{a2} with $\mu^k \equiv \mu$, $\delta^k\equiv \delta$, $\beta=0$, and $x^0=p^0=0$,  converges to the unique minimizer $\hat{x}$.
\end{corollary}

\subsection{Convergence for nonconvex optimization}
This part is about the convergence analysis of LBreIF for minimizing a nonconvex objective function $E(x)$. We start with a sufficient descent lemma, which generalizes the central result--Lemma 4.2 in \cite{benning2021choose}.
\begin{lemma}[Sufficient descent]\label{mainlemma01}
Let $\{x^k\}$ be the sequence generated by \eqref{a1}-\eqref{a2} with $\inf{\mu^k}\geq \mu >0$ and the stepsize $\delta^k$ satisfying \eqref{p1}. Denote
$$\rho:=\frac{L\omega}{1+\alpha(h)-\omega},~~ \omega\in (0,1+\alpha(h)).$$
If Assumption \ref{ass} holds, then we have the following sufficient decrease property:
\begin{equation}\label{sdp}
f_\beta(x^{k+1})+\rho D_h(x^{k+1},x^k)+\mu D_R^{symm}(x^{k+1},x^k)\leq f_\beta(x^k).
\end{equation}
In particular, we observe that
\begin{equation}\label{lim}
\lim_{k\rightarrow \infty} D_h(x^{k+1},x^k)=\lim_{k\rightarrow \infty} D_R^{symm}(x^{k+1},x^k)=0.
\end{equation}
Here, $D_R^{symm}$ is the symmetric generalized Bregman distance, defined as
$$D_R^{symm}(u,v):=D_R^q(u,v)+D_R^p(v,u)=\langle p-q, u-v\rangle$$
for $u,v\in \dom R$ with $p\in \partial R(u)$ and $q\in \partial R(v)$.
\end{lemma}
\begin{proof}
Using the inequality \eqref{desineq0} in Lemma \ref{dein}, the definition of $D_R^{symm}$, and the condition \eqref{p1} on $\delta^k$, we derive that
\begin{eqnarray} \label{meq01}
\begin{array}{lll}
\delta^k(f_\beta(x^{k+1})-f_\beta(x^k))&\leq & -\mu^k\delta^kD_R^{symm}(x^{k+1},x^k) -D_h(x^k,x^{k+1})-(1-\delta^k L)D_h(x^{k+1},x^k)\\
&\leq & -\mu^k\delta^kD_R^{symm}(x^{k+1},x^k) -D_h(x^k,x^{k+1})+(\alpha -\omega)D_h(x^{k+1},x^k) \\
&\leq & -\mu^k\delta^kD_R^{symm}(x^{k+1},x^k)-\omega D_h(x^{k+1},x^k),
\end{array}
\end{eqnarray}
where the last relationship follows from the definition of $\alpha(h)$. Thus, \eqref{sdp} follows from \eqref{meq01} by rearranging the terms and using the conditions on $\delta^k$ and $\mu^k$ and also the notation $\rho$.

To obtain \eqref{lim}, one can sum \eqref{sdp} over $k=0, \cdots, n$ to get that
$$\sum_{k=0}^n ( \rho D_h(x^{k+1},x^k)+\mu D_R^{symm}(x^{k+1},x^k))\leq f_\beta(x^0)-f_\beta(x^n)\leq  f_\beta(x^0)- \inf_{x\in\dom h}f_\beta(x).$$
Therefore, $\sum_{k=0}^\infty D_h(x^{k+1},x^k)<\infty$ and $\sum_{k=0}^\infty D_R^{symm}(x^{k+1},x^k)<\infty$. Hence, \eqref{lim} follows immediately. This completes the proof.
\end{proof}

The set of all limit points of $\{x^k\}$ is denoted by $\Omega$. In other words,
\begin{eqnarray*}
\begin{array}{ll}
\Omega:=&\{x^*\in\EE: \textrm{there exists an increasing sequence}\\
 &\textrm{ of integers} ~\{{k_i}\} ~\textrm{such that} ~ \lim_{i\rightarrow \infty}x^{k_i}=x^*\}.
\end{array}
\end{eqnarray*}

\begin{lemma}[Point and function value convergence]\label{mainlemma02}
In addition to Assumption \ref{ass} and (ii) of Assumption \ref{ass1}, we assume that $h$ is strongly convex on $\dom h$ with $\overline{\dom h}=\dom h$ and that the level set  $\{x: f_\beta(x)\leq f_\beta(x^0)\}$ is bounded. Then, we have
$\Omega\neq \emptyset  $ and for any limit point $x^*\in \Omega$, $$\lim_{k\rightarrow \infty}f_\beta(x^k)=f_\beta(x^*).$$
\end{lemma}

\begin{proof}

The boundedness of $\{x: f_\beta(x)\leq f_\beta(x^0)\}$ and the nonincreasing property of $\{f_\beta(x^k)\}$ from \eqref{sdp} ensure the boundedness of $\{x^k\}$. Hence, $\emptyset \neq \Omega $. Take $x^*\in \Omega$. This means there exists a subsequence $\{x^{k_i}\}\subset \{x^k\}\subset \rint\dom h$ such that $\lim_{i\rightarrow \infty}x^{k_i}=x^*\in \overline{\dom h}= \dom h$. Together with \eqref{lim} in Lemma \ref{mainlemma01} and using the strong convexity of $h$, we can conclude that as $i\rightarrow \infty$,
 \begin{equation}\label{lim01}
~ D_h(x^{k_i+1},x^{k_i})\rightarrow 0, ~\|x^{k_i+1}-x^{k_i}\|\rightarrow 0, ~D_R(x^{k_i+1},x^{k_i})\rightarrow 0.
 \end{equation}
Note that $\{x^k\}\subset \rint\dom h\subset \rint\dom E$ due to (iii) of Assumption \ref{ass}. Using boundedness $\{x^k\}$ and Theorem 3.16 in \cite{Beck2017First}, we know that $\{\nabla E(x^k)\}$ is bounded.
Hence, as $i\rightarrow \infty$,
\begin{eqnarray}\label{lim02}
\begin{array} {lll}
 &&\langle \nabla E(x^{k_i}), x^*-x^{k_i}\rangle \leq \|\nabla  E(x^{k_i})\| \|x^*-x^{k_i}\| \rightarrow 0,\\
 &&\langle \nabla E(x^{k_i}), x^{k_i+1}-x^{k_i}\rangle \leq \|\nabla  E(x^{k_i})\| \|x^{k_i+1}-x^{k_i}\| \rightarrow 0.
\end{array}
\end{eqnarray}
In light of \eqref{a1}, we have
\begin{eqnarray} \label{meq02}
\begin{array}{lll}
& & \langle \nabla E(x^{k_i}), x^{k_i+1}-x^{k_i}\rangle+ \beta R(x^{k_i+1})+\frac{1}{\delta^{k_i}}D_h(x^{k_i+1},x^{k_i}) + \mu^{k_i} D_R (x^{k_i+1},x^{k_i}) \\
&\leq & \langle \nabla E(x^{k_i}), x^*-x^{k_i}\rangle+ \beta R(x^*)+\frac{1}{\delta^{k_i}}D_h(x^*,x^{k_i}) + \mu^{k_i} D_R (x^*,x^{k_i}).
\end{array}
\end{eqnarray}
Letting $i\rightarrow\infty$ in the above inequality, using the results \eqref{lim01}-\eqref{lim02} and the assumption that $D_h(y, y^k)\rightarrow 0$ and $D_R(y, y^k)\rightarrow 0$ if $\{y^k\}\subset \rint\dom h$ converges to some $y\in \dom h$, we obtain
 \begin{equation}\label{meq03}
 {\lim\sup}_{i\rightarrow \infty}R(x^{k_i+1})\leq R(x^*).
  \end{equation}
Combining the continuity of $E$ over $\dom h=\overline{\dom h}$ and noting that $x^{k_i+1}\rightarrow x^*$ as $i\rightarrow \infty$ as well since $\|x^{k_i+1}-x^{k_i}\|\rightarrow 0$, we further have
 \begin{equation}\label{meq04}
 {\lim\sup}_{i\rightarrow \infty} (E(x^{k_i+1})+\beta R(x^{k_i+1}))\leq E(x^*)+\beta R(x^*).
  \end{equation}
On the other hand, using the lower semicontinuity of $R$ and $E$, we derive that
\begin{eqnarray} \label{meq05}
\begin{array}{lll}
{\lim\inf}_{i\rightarrow \infty} (E(x^{k_i+1})+\beta R(x^{k_i+1}))&\geq & {\lim\inf}_{i\rightarrow \infty} E(x^{k_i+1}) +{\lim\inf}_{i\rightarrow \infty} \beta R(x^{k_i+1}) \\
&\geq & E(x^*)+\beta R(x^*).
\end{array}
\end{eqnarray}
Therefore, we get
 \begin{equation}\label{meq06}
  {\lim}_{i\rightarrow \infty}f_\beta(x^{k_i+1})= {\lim}_{i\rightarrow \infty} (E(x^{k_i+1})+\beta R(x^{k_i+1}))= E(x^*)+\beta R(x^*)=f_\beta(x^*).
  \end{equation}
Note that $\{f_\beta(x^k)\}$ is a nonincreasing sequence from Lemma \ref{mainlemma01} and is lower bounded by the finite value $\inf_{x\in\dom h}f_\beta(x)$ and hence it is convergent. Therefore, we have  $\lim_{k\rightarrow \infty}f_\beta(x^k)=f_\beta(x^*).$
\end{proof}

In order to derive the global convergence of $\{x^k\}$ without the gradient Lipschitz continuity and convexity of $E$, we combine the method in \cite{benning2021choose} and that in \cite{bolte2017first}, both of which were originally inspired by \cite{2014Proximal}. First, we introduce a modified surrogate function $F:\EE \times \EE\rightarrow \RR\bigcup\{+\infty\}$ in the same spirit of that in \cite{benning2021choose}:
  \begin{equation}\label{sof}
F_{\beta,\mu}(x,y):=E(x)+(\beta+\mu)R(x)+\mu R^*(y)-\mu\langle x, y\rangle.
  \end{equation}
Note that the case of $\beta=0$ and $\mu=1$ reduces to the surrogate function in \cite{benning2021choose}, which is only defined for functions with a Lipschitz continuous gradient. In what follows, we will deduce a sufficient decrease property of the surrogate function $F_{\beta,\mu}$ and its subgradient bound. To this end, we first link this function with the known function $f_\beta$ and some Bregman distance. For any fixed $y\in \dom \partial R$, we take $z\in \partial R^*(y)$. Then by the Fenchel theorem, we have $y\in \partial R(z)$ and $R^*(y)+R(z)=\langle y, z\rangle$, with which the surrogate function $F_{\beta,\mu}$ now can be reformulated as
  \begin{equation}\label{sof1}
F_{\beta,\mu}(x,y) =f_\beta(x)+\mu D_R(x, z).
  \end{equation}
  In particular, we have $F_{\beta,\mu}(x^{k+1},p^k) =f_\beta(x^{k+1})+\mu D_R(x^{k+1}, x^k)$ since $p^k\in \partial R(x^k)$ and hence $x^k\in \partial R^*(p^k)$. Note that
   \begin{equation}\label{sofsg}
\partial F_{\beta,\mu}(x,y) =\left\{\left(\begin{array}{c}
                                      \nabla E(x)+(\beta+\mu)s_1-\mu y \\
                                      \mu s_2-\mu x
                                    \end{array}\right): s_1\in\partial R(x), s_2\in \partial R^*(y)\right\}.
  \end{equation}
Using again the fact that $p^{k+1}\in \partial R(x^{k+1})$ and $x^k\in \partial R^*(p^k)$, we can get
   \begin{equation}\label{sg}
r^{k+1} :=\left(\begin{array}{c}
                                      \nabla E(x^{k+1})+(\beta+\mu)p^{k+1}-\mu p^k \\
                                      \mu x^{k}-\mu x^{k+1}
                                    \end{array}\right)\in \partial F_{\beta,\mu}(x^{k+1},p^k).
  \end{equation}
Denote $(s^{k}):=(x^{k}, p^{k-1})$. The set of all limit points of $\{s^k\}$ is denoted by $\Omega_0$. In other words,
\begin{eqnarray*}
\begin{array}{ll}
\Omega_0:=&\{s^*:=(x^*,p^*)\in\EE\times \EE: \textrm{there exists an increasing sequence}\\
 &\textrm{ of integers} ~\{{k_i}\} ~\textrm{such that} ~ \lim_{i\rightarrow \infty}x^{k_i}=x^*, \lim_{i\rightarrow \infty}p^{k_i-1}=p^*\}.
\end{array}
\end{eqnarray*}
We point out that the set $\Omega_0$ above is slightly different from the set of limit points in \cite{benning2021choose}, where the following definition was used.
\begin{eqnarray*}
\begin{array}{ll}
\omega(s^0):=&\{s^*:=(x^*,p^*)\in\EE\times \EE: \textrm{there exists an increasing sequence}\\
 &\textrm{ of integers}~ \{{k_i}\} ~ \textrm{such that}~  \lim_{i\rightarrow \infty}x^{k_i}=x^*, \lim_{i\rightarrow \infty}p^{k_i}=p^*\}.
\end{array}
\end{eqnarray*}

\begin{lemma} \label{mainlemma05}
Under the same setting as Lemma \ref{mainlemma01}, we have the following sufficient decrease property
\begin{equation}\label{sdp00}
F_{\beta,\mu}(x^{k+1}, p^k)+\rho D_h(x^{k+1},x^k) + \mu D_R(x^k,x^{k+1})+\mu D_R(x^k,x^{k-1})\leq F_{\beta,\mu}(x^{k}, p^{k-1}).
\end{equation}
Suppose further that Assumption \ref{ass3} holds for $h$ and $E$ and that the level set  $\{x: f_\beta(x)\leq f_\beta(x^0)\}$ is bounded. Then, we have the subgradient bound by the iterates gap
\begin{equation}\label{sb}
\|r^{k+1}\|\leq \rho_2 \|x^{k+1}-x^k\|+(\mu^k-\mu)\|p^{k+1}-p^k\|,
\end{equation}
where $\rho_2:=L_E+\frac{L_f}{\delta}+\mu$ and $\delta:=\sup_k\{\delta^k\}$.
Moreover, if (ii)of Assumption \ref{ass1} and Assumption \ref{ass2} also hold, $\lim_k\mu^k=\mu$, and we also assume that $h$ is strongly convex on $\dom h$ with $\overline{\dom h}=\dom h$, then $\Omega_0$ must be a nonempty and compact set, and for every $s^*=(x^*,p^*)\in \Omega_0$ we have $\lim_{k\rightarrow \infty}\dist(s^k, \Omega_0)=0$ and
\begin{equation}\label{lim2}
\lim_{k\rightarrow \infty}F_{\beta,\mu}(s^k) =f_\beta(x^*).
\end{equation}
\end{lemma}

\begin{proof}
Adding $\mu D_R(x^k, x^{k-1})$ to both sides of \eqref{sdp} and using the formulation \eqref{sof1}, we deduce \eqref{sdp00}. The boundedness of the level set  $\{x: f_\beta(x)\leq f_\beta(x^0)\}$  and the monotonicity of $\{f_\beta(x^k)\}$ ensure the boundedness of $\{x^k\}$ and hence Assumption \ref{ass3} can be employed to bound the difference of gradient below. Actually, we can derive that
\begin{eqnarray} \label{bd}
\begin{array}{lll}
\|r^{k+1}\|&\leq &  \|\nabla E(x^{k+1})+(\beta+\mu)p^{k+1}-\mu p^k\|+ \mu\|x^{k}- x^{k+1}\|\\
&\leq & \|\nabla E(x^{k+1})+(\beta+\mu^k)p^{k+1}-\mu^k p^k\|+(\mu^k-\mu)\|p^{k+1}-p^k\|+ \mu\|x^{k}- x^{k+1}\|\\
&=& \|\nabla E(x^{k+1})-\nabla E(x^k)-  \frac{1}{\delta^k}(\nabla h(x^{k+1})-\nabla h(x^k) ) \|\\
& &+(\mu^k-\mu)\|p^{k+1}-p^k\|+ \mu\|x^{k}- x^{k+1}\|\\
&\leq & \|\nabla E(x^{k+1})-\nabla E(x^k)\|+ \frac{1}{\delta^k}\|\nabla h(x^{k+1})-\nabla h(x^k)  \|\\
& &+(\mu^k-\mu)\|p^{k+1}-p^k\|+ \mu\|x^{k}- x^{k+1}\|\\
&\leq & (L_E+\frac{L_f}{\delta}+\mu)\| x^{k+1}-x^{k}\|+(\mu^k-\mu)\|p^{k+1}-p^k\|,
\end{array}
\end{eqnarray}
where the  equality follows from \eqref{a2}.

Now, we show the nonemptyness of $\Omega_0$. By the boundedness of $\{x^k\}$, there exists an increasing of integers $\{i_j\}_{j\in\mathbb{N}}$ such that $\lim_{j\rightarrow \infty} x^{i_j}=x^*$. Recall that $p^{i_j}\in\partial R(x^{i_j})$. Using the locally bounded subgradient Assumption \ref{ass2}, we know that  $\{p^{i_j}\}$ must be bounded (actually $\{p^k\}$ is bounded due to the same argument) and hence there exists a subsequence $\{k_i\}\subset \{i_j\}$ such that $\lim_{i\rightarrow \infty}p^{k_i}=\bar{p}$. From \eqref{a2}, it holds that
\begin{equation}
(\mu^{k_i-1}+\beta)p^{k_i}=\mu^{k_i-1} p^{k_i-1} - \nabla E(x^{k_i-1})- \frac{1}{\delta^{k_i-1}}(\nabla h(x^{k_i})-\nabla h(x^{k_i-1})).
\end{equation}
Note that $\lim_i x^{k_i}=\lim_i x^{k_i-1}=x^*$, $\lim_i \mu^{k_i-1}=\mu$ and $\{\delta^{k_i-1}\}$ is bounded. Together with Assumption \ref{ass3}, we conclude that there exists a point $p^*$ such that $\lim_{i\rightarrow \infty}p^{k_i-1}=p^*$ (such point may be different from $\bar{p}$). Therefore, $s^*=(x^*, p^*)$ indeed belongs to $\Omega_0$ and hence it is nonempty. In particular, $x^*\in \Omega$ for each $s^*=(x^*, p^*)\in\Omega_0$. Thus, in light of Lemma \ref{mainlemma01} and Lemma \ref{mainlemma02}, we derive that
\begin{equation}
 \lim_{k\rightarrow \infty}F_{\beta,\mu}(s^k) =\lim_{k\rightarrow \infty}F_{\beta,\mu}(x^{k},p^{k-1}) =\lim_{k\rightarrow \infty}(f_{\beta}(x^{k})+\mu D_R(x^{k},x^{k-1}))=f_\beta(x^*).
\end{equation}
From Theorem 3.7 in \cite{Rudin}, we know that the set $\Omega_0$ must be closed since it is the set of cluster points of $\{s^k\}$. The boundedness of $\Omega_0$ is due to the boundedness of $\{x^k, p^k\}$. Therefore, the set $\Omega_0$ is compact and hence $\lim_{k\rightarrow \infty}\dist(s^k, \Omega_0)=0$ by the definition of limit points. This completes the proof.
\end{proof}

\begin{lemma}\label{lemma11}
Let $\{a^k\}$ and $\{b^k\}$ be given sequences of $\EE$. If $a^k\rightarrow 0$ and $b^k\rightarrow b\neq 0$ as $k\rightarrow \infty$, then as $n\rightarrow \infty$ we have
\begin{equation}
\|\sum_{k=0}^n(a^k+b^k)\|\rightarrow \infty.
\end{equation}
\end{lemma}
\begin{proof}
Using the condition that $a^k\rightarrow 0$ and $b^k\rightarrow b\neq 0$ as $k\rightarrow \infty$, we can find an index $k_0$ such that for any $k\geq k_0$, it holds that
\begin{equation}
\|a^k\|<\frac{1}{4}\|b\|, ~~\|b^k-b\|<\frac{1}{4}\|b\|.
\end{equation}
Letting $n>k_0$ and using the reverse triangle inequality $\|u+\sum_iu_i\|\geq \|u\|-\sum_i\|u_i\|$, we drive that
\begin{eqnarray}
\begin{array}{lll}
\|\sum_{k=0}^n(a^k+b^k)\|& = &  \|\sum_{k=0}^nb+ \sum_{k=0}^n(a^k+(b^k-b))\|\\
& \geq &  \|\sum_{k=0}^n b\| - \sum_{k=0}^n\|a^k\| -\sum_{k=0}^n\|b^k-b\|\\
& \geq & (n+1)\|b\| - \sum_{k=0}^{k_0}(\|a^k\| +\|b^k-b\|) -\frac{n-k_0}{2}\|b\|\\
&=& \frac{n+2+k_0}{2}\|b\|- \sum_{k=0}^{k_0}(\|a^k\| +\|b^k-b\|),
\end{array}
\end{eqnarray}
from which the conclusion follows. This completes the proof.
\end{proof}

Now, we are ready to present the second main convergence theorem in this study.
\begin{theorem}\label{mainthm2}
Suppose that $F_{\beta, \mu}$ is a KL function, Assumptions \ref{ass}, \ref{ass2} and (ii) of Assumption \ref{ass1} hold,  Assumption \ref{ass3} holds for $h$ and $E$,  the level set  $\{x: f_\beta(x)\leq f_\beta(x^0)\}$ is bounded, and $h$ is strongly convex with $\overline{\dom h}=\dom h$. Let the parameters $\delta^k$ and $\mu^k$ satisfy \eqref{p1}-\eqref{p2} and $\sum_{k=0}^\infty(\mu^k-\mu)<\infty$. Let $\{(x^k,p^k)\}$ be the sequence generated by \eqref{a1}-\eqref{a2}. Then, the sequence $\{x^k\}$ has finite length in the sense that
\begin{equation}\label{cau}
 \sum_{k=0}^\infty \|x^{k+1}-x^k\|<\infty.
\end{equation}
Moreover, the sequence $\{x^k\}$ converges to a critical point $x^*$ of $E$ in the sense that $\nabla E(x^*)=0$ if $\beta=0$. Furthermore, if $R^*$ is assumed to be locally strongly convex, then the dual sequence $\{p^k\}$ also converges and the limit point $x^*$ of $\{x^k\}$ is a critical point of $f_\beta$ in the sense that $0\in \nabla E(x^*) +\beta \partial R(x^*)$.
\end{theorem}

\begin{proof}
We divide the proof into two parts. The first part is to show \eqref{cau} by modifying the methodology in \cite{2014Proximal}.
Let us begin with any point $s^*=(x^*, p^*)\in \Omega_0$. Then, there exists an increasing sequence of integers $\{k_i\}_{i\in\mathbb{N}}$ such that $x^{k_i}\rightarrow x^*$ as $i\rightarrow \infty$. From Lemma \ref{mainlemma05} and recalling that $s^{k}=(x^k, p^{k-1})$, we know
\begin{equation}
\lim_{k\rightarrow \infty}F_{\beta,\mu}(s^k) =f_\beta(x^*).
\end{equation}
Note that the convergent sequence $\{F_{\beta,\mu}(s^k)\}$ is nonincreasing.  If there exists an integer $\bar{k}$ such that $F_{\beta,\mu}(s^{\bar{k}}) =f_\beta(x^*)$, then $F_{\beta,\mu}(s^k)\equiv f_\beta(x^*)$ for $k\geq \bar{k}$ and hence $D_h(x^{k+1},x^k)=0$ for $k\geq \bar{k}$ from \eqref{sdp00},
which implies that $x^k\equiv x^{\bar{k}}$ for $k\geq \bar{k}$ due to the strong convexity of $h$. Hence, the result \eqref{cau} follows trivially.
If there does not exist such an index, then it must hold that $F_{\beta,\mu}(s^k)  >f_\beta(x^*)$ holds for all $k>0$.
Since $\lim_{k\rightarrow \infty}F_{\beta,\mu}(s^k) =f_\beta(x^*)$, for any $\eta>0$ there must exist an integer $\hat{k}>0$ such that $F_{\beta,\mu}(s^k)<f_\beta(x^*)+\eta$ for all $k>\hat{k}$. Similarly, $\lim_{k\rightarrow \infty} \dist (s^k, \Omega_0)=0$ implies for any $\zeta>0$ there must exist an integer $\widetilde{k}>0$ such that $\dist(s^k, \Omega_0)<\zeta$ for all $k>\widetilde{k}$.
Therefore,  for all $k>l:=\max\{\hat{k},\widetilde{k}\}$ we have
 \begin{equation}
 s^k\in \{s:\dist(s,\Omega_0)<\zeta\}\bigcap \{s: f_\beta(x^*)<F_{\beta,\mu}(s)<f_\beta(x^*)+\eta\}.
\end{equation}
Thus, we apply Lemma \ref{UKL} to get,
 \begin{equation}\label{ukl1}
  \varphi^{\prime}(F_{\beta,\mu}(s^k)-f_\beta(x^*))\dist(0,\partial F_{\beta,\mu}(s^k))\geq 1.
\end{equation}
Recall that  $r^k\in \partial F_{\beta,\mu}(s^k)$.  Using \eqref{sb} in Lemma \ref{mainlemma05}, we get that
\begin{equation}\label{ukl2}
 \textrm{dist}(0,\partial F_{\beta,\mu}(s^k))\leq \|r^{k}\|\leq \rho_2 \|x^{k}-x^{k-1}\|+(\mu^{k-1}-\mu)\|p^{k}-p^{k-1}\|.
\end{equation}
On the other hand, from the concavity of $\varphi$ we know that
$$\varphi^\prime(x)\leq \frac{\varphi(x)-\varphi(y)}{x-y}$$
holds for all $x,y\in [0, \eta), x>y$. Hence, by taking $x= F_{\beta,\mu}(s^k)-f_\beta(x^*)$ and
$y=F_{\beta,\mu}(s^{k+1})-f_\beta(x^*)$ in the inequality above, we get
\begin{equation}\label{ukl3}
\varphi^\prime(F_{\beta,\mu}(s^k)-f_\beta(x^*))\leq \frac{\varphi^k-\varphi^{k+1}}{F_{\beta,\mu}(s^k)-F_{\beta,\mu}(s^{k+1})}\leq \frac{\varphi^k-\varphi^{k+1}}{\rho\nu\|x^{k+1}-x^k\|^2},
\end{equation}
where $\varphi^k:=\varphi(F_{\beta,\mu}(s^k)-f_\beta(x^*))$ and the last inequality follows from \eqref{sdp00} and the strong convexity $D_h(x^{k+1},x^k)\geq \nu \|x^{k+1}-x^k\|^2$ for some constant $\nu>0$. Therefore, from \eqref{ukl1}-\eqref{ukl3} we get
$$\|x^{k+1}-x^k\|^2\leq \frac{\rho_2}{\rho\nu}\left(\varphi^k-\varphi^{k+1}\right)\left( \|x^{k}-x^{k-1}\|+\frac{\mu^{k-1}-\mu}{\rho_2}\|p^{k}-p^{k-1}\|\right)$$
Based on the Young's inequality of the form $2\sqrt{ab}\leq a+b$, we further get
$$2\|x^{k+1}-x^k\| \leq \frac{\rho_2}{\rho\nu}(\varphi^k-\varphi^{k+1})+  \|x^{k}-x^{k-1}\|+\frac{\mu^{k-1}-\mu}{\rho_2}\|p^{k}-p^{k-1}\|$$
Summing the inequality above from $k=l,\cdots, N$, we deduce
\begin{align}
2\sum_{k=l}^N \|x^{k+1}-x^k\| & \leq \sum_{k=l}^N \|x^k-x^{k-1}\| + \frac{\rho_2}{\rho\nu}(\varphi^l-\varphi^{N+1}) + \sum_{k=l}^N \frac{\mu^{k-1}-\mu}{\rho_2}\|p^{k}-p^{k-1}\| \nonumber\\
 & \leq \sum_{k=l}^N \|x^{k+1}-x^k\| + \|x^{l}-x^{l-1}\|+ \frac{\rho_2}{\rho\nu}\varphi^l + \sum_{k=l}^N \frac{\mu^{k-1}-\mu}{C\rho_2}.
\end{align}
where the second inequality follows by using the boundedness of $\{p^k\}$, say $\|p^{k}-p^{k-1}\|\leq \frac{1}{C}$ for some constant $C>0$.
Thus, we have
$$\sum_{k=l}^N \|x^{k+1}-x^k\| \leq \|x^{l}-x^{l-1}\|+ \frac{\rho_2}{\rho\nu}\varphi^l + \sum_{k=l}^N \frac{\mu^{k-1}-\mu}{C\rho_2}<\infty,$$
which immediately implies the result \eqref{cau}.

Now, we turn into the second part  to analyze the convergence of $\{x^k\}$ and $\{p^k\}$. Let us first show that $\{x^k\}$ is a Cauchy sequence and hence it converges. In fact, the finite length property implies that $\sum_{k=l}^\infty \|x^{k+1}-x^k\|\rightarrow 0$ as $l\rightarrow \infty$. Thus, for any $m > n\geq l$ we have
$$\|x^m-x^n\|=\|\sum_{k=n}^{m-1}(x^{k+1}-x^k)\|\leq \sum_{k=n}^{m-1}\|x^{k+1}-x^k\|\leq \sum_{k=l}^\infty \|x^{k+1}-x^k\|,$$
 which implies that $\{x^k\}$ is a Cauchy sequence. Using \eqref{a2} with $\beta=0$, we get
 \begin{equation}\label{pp}
p^k- p^{k+1}=\frac{1}{\delta^k\mu^k}(\nabla h(x^{k+1})-\nabla h(x^k)) +\frac{1}{\mu^k} \nabla E(x^k).
\end{equation}
Summing \eqref{pp} over $k=0,\cdots, n$, we get
 \begin{equation}\label{pp1}
p^0- p^{n+1}=\sum_{k=0}^n\left(\frac{1}{\delta^k\mu^k}(\nabla h(x^{k+1})-\nabla h(x^k)) +\frac{1}{\mu^k}\nabla E(x^k)\right).
\end{equation}
Assume that $\nabla E(x^*)\neq 0$. Noting that $\frac{1}{\delta^k\mu^k}(\nabla h(x^{k+1})-\nabla h(x^k))\rightarrow 0$ and $\frac{1}{\mu^k} E(x^k)\rightarrow \frac{1}{\mu}\nabla E(x^*)\neq 0$, we invoke Lemma \ref{lemma11} to conclude that
$\|p^0- p^{n+1}\|\rightarrow \infty$ as $n\rightarrow \infty$, which contradicts the boundedness of $\{p^k\}$. Therefore, we have $\nabla E(x^*)= 0$.

If $R^*$ is locally strongly convex, then for some constant $\nu_1>0$,
$$D_R(x^k,x^{k-1})=D_{R^*}(p^{k-1},p^k)\geq \frac{\nu_1}{2}\|p^k-p^{k-1}\|^2.$$
Recall that $D_h(x^{k+1},x^k)\geq \nu\|x^{k+1}-x^k\|^2$. Thus, using \eqref{sdp00} and letting $\nu_2=\min\{\frac{\mu\nu_1}{2},\rho\nu\}$, we get
 \begin{equation}\label{sdn1}
F_{\beta,\mu}(s^{k+1})+\nu_2\|s^{k+1}-s^k\|^2\leq F_{\beta,\mu}(s^{k}).
\end{equation}
From \eqref{sb}, since $\|x^{k+1}-x^k\|\leq \|s^{k+1}-s^k\|$ we deduce
\begin{equation}\label{sg}
\|r^{k+1}\|\leq \rho_2 \|s^{k+1}-s^k\|+(\mu^k-\mu)\|p^{k+1}-p^k\|,
\end{equation}
Now we repeat the argument of the first part to conclude that $\{s^k\}$ has a finite length and hence it converges. Hence, $\{p^k\}$ also converges to some $p^*$. Note that \eqref{a2} has the following form
$$\nabla E(x^k)+\beta p^k=(\mu^k+\beta)(p^k-p^{k+1})-\frac{1}{\delta^k}(\nabla h(x^{k+1})-\nabla h(x^k)).$$
Letting $k\rightarrow \infty$ above, we immediately get $0=\nabla E(x^*)+\beta p^*$. Finally, using the subgradient inequality and the lsc property of $R$, we drive that for any $x\in \dom R$
\begin{align}
R(x)={\lim\inf}_{k\rightarrow \infty}R(x) & \geq {\lim\inf}_{k\rightarrow \infty}(R(x^k)+\langle p^k,x-x^k\rangle) \nonumber\\
 & \geq R(x^*)+ \langle p^*,x-x^*\rangle,
\end{align}
which implies that $p^*\in \partial R(x^*)$. Thus, we finally get $0\in \nabla E(x^*)+\beta \partial R(x^*).$ This completes the proof.

\end{proof}

\section{Application to inverse problems}\label{se5}
In this section, we introduce two examples to show how our algorithmic framework and its convergence theory can be applied.

\subsection{A convex example: Linear inverse problems}
We have described the linear inverse problems in the introduction. Here, we apply the LBreIF to linear inverse problems with nonnegative data $b\in \RR^m$, considered in \cite{Bauschke2016A} in the following form
\begin{equation}\label{lip}
\min_{x\in\RR^n}\{\Phi(x):=D_\phi(Ax, b)+\lambda R(x)\},
\end{equation}
 where the distance $D_\phi(Ax, b)$ with $\phi(x)=\sum_{i=1}^nx_i\log x_i$ is adopted to measure the ``error" between $b$ and $Ax$, $R(\cdot)$ is a regularizer reflecting prior information on the solution, the parameter $\lambda>0$ balances the data fidelity and the solution regularization.  In order to apply our method and theory, we let
 $E(x)=D_\phi(Ax, b)+\frac{\epsilon}{2}\|Ax-b\|^2$ and $h(x)=\sum_{i=1}^nx_i\log x_i+\frac{\epsilon}{2}\|x\|^2$. It should be noted that we here add the terms  $\frac{\epsilon}{2}\|Ax-b\|^2$ and $\frac{\epsilon}{2}\|x\|^2$ with $\epsilon \geq 0$ to the distance $D_\phi(Ax, b)$ and the Boltzmann-Shannon entropy $\sum_{i=1}^nx_i\log x_i$, respectively. For the former, we aim to meet the condition \eqref{qg} in Theorem \ref{mainthm1}. Actually, in this setting $E(u,b)=D_\phi(u,b)+\frac{\epsilon}{2}\|u-b\|^2$ is $\epsilon$-strongly convex and and its minimizer $\hat{u}$ is attainable.
Hence,
 $$E(u,b)-E(\hat{u},b)\geq \frac{\epsilon}{2}\|u-\hat{u}\|^2.$$
 For the latter, we aim to meet the (LC) condition. Indeed, due to Lemma 8 in \cite{Bauschke2016A}, the (LC) condition holds for the pair of $(E,h)$ with
$$L\geq \max\{\|A\|^2, \max_{1\leq j\leq n}\sum_{i=1}a_{ij}\}.$$
Since $\dom (x_i\log x_i)=[0,\infty)$, we can conclude that $\overline{\dom h}=\dom h$. From the expression of $\nabla h(x)= e +\log x +\epsilon x$ where $e$ stands for the vector whose each entry equals to one, we know that $h$ is continuously differentiable on $\rint\dom h$. If $A$ is surjective, then $Ax=b$ is a consistence system and hence $\cX_0=\arg\min \{E(x): x\in \overline{\dom h}\}$ must be nonempty. Since $R(x)=\|x\|_1$ is real-valued convex, Assumptions \ref{ass}-\ref{ass2} about $R$ can be easily verified to hold. Regarding to the computation, we can reduce  \eqref{a1} to get
\begin{eqnarray}\label{more1}
	x^{k+1}
	= \arg\min_x \{\delta \mu \|x\|_1 + \langle \overline{p}^k, x\rangle + \sum_{i=1}^nx_i\log x_i+\frac{\epsilon}{2}\|x\|^2 \},			
\end{eqnarray}
where $\overline{p}^k = \delta \nabla E(x^k) - \nabla h(x^k) - \delta \mu p^k$. Thus, the entries $x^{k+1}_i$ are the roots of the equations $\log x_i + \epsilon x_i + C_i = 0, i = 1,...,m$
with the constants $C_i = \delta \mu + \overline{p}^k_i + 1$.
As a consequence of Theorem \ref{mainthm1}, we have the following convergence result for the generated sequence $\{x^k\}$.
\begin{corollary}
Let $E$, $h$ and $R$ be given above with a surjective $A$.
Then, the sequence $\{x^k\}$, generated by \eqref{more1} with $x^0=p^0=0$,  converges either to the unique minimizer $\hat{x}:=\arg\min\{\mu\|x\|_1+\frac{\epsilon}{2\delta}\|x\|^2+\frac{1}{\delta}\sum_{i=1}^nx_i\log x_i: x\in \cX_0\}$ or to a boundary point of $\dom h$.
\end{corollary}

\subsection{A nonconvex example: Quadratic inverse problems}
We briefly describe the quadratic inverse problems as follows. Given a finite number of symmetric matrices $A_i\in \RR^{d\times d},~i=1,2\cdots, m$, modeling the measure procedure, and a vector $b\in\RR^m$ recording the measure data $(b_1,b_2,\cdots, b_m)$, the goal is to find $x\in \RR^d$ such that
\begin{equation}\label{qip}
x^TA_ix\simeq b_i, ~i=1,2\cdots, m.
\end{equation}
As a natural extension of the linear inverse problems, the quadratic inverse problems arise in the broad area of signal processing, including for example the phase retrieval problems \cite{2017Phase} as special cases. Similar to the linear inverse problems, the system of quadratic equations \eqref{qip} is usually underdetermined. Thus, there also needs the regularization technique to help find the ``right" solution via solving the regularized optimization problem
\begin{equation}\label{rqip}
\min_{x\in\RR^d}\{\Psi(x):=\frac{1}{4}\sum_{i=1}^m(x^TA_ix-b_i)^2+\lambda R(x)\},
\end{equation}
where the least-squares models the ``error" between $b_i$ and $x^TA_ix$, $R(\cdot)$ is a regularizer reflecting prior information on the solution, the parameter $\lambda>0$ balances the data fidelity and the solution regularization. Instead of solving \eqref{rqip}, we consider the LBreIF for \eqref{qip} with sparse prior. To this end, we let $h(x):=\frac{1}{4}\|x\|^4+\frac{1}{2}\|x\|^2$, $R(x)=\|x\|_1$, and
$$E(x):=\frac{1}{4}\sum_{i=1}^m(x^TA_ix-b_i)^2+\frac{\epsilon}{2}\|x\|^2.$$
Here, the term $\frac{\epsilon}{2}\|x\|^2$ with $\epsilon >0$ is added to ensure the level boundedness of $E(x)$. Due to Lemma 5.1 in \cite{bolte2017first}, the (LC) condition holds for the pair of $(E,h)$ with
$$L\geq \sum_{i=1}^m (3\|A_i\|^2+\|A_i\|\cdot|b_i|)+\epsilon.$$
Note that
$$\nabla h(x)=(\|x\|^2+1)x, ~~\nabla E(x)=\sum_{i=1}^m(x^TA_ix-b_i)A_ix+\epsilon x.$$
Now, the iterate sequence $\{x^k\}$, based on LBreIF, is defined via solving
\begin{eqnarray}\label{more2}
	x^{k+1}
    := \arg\min_x \{\delta^k \mu^k \|x\|_1 + \langle \overline{p}^k, x\rangle + \frac{1}{4}\|x\|^4+\frac{1}{2}\|x\|^2 \},				
\end{eqnarray}
where $\overline{p}^k = \delta^k \nabla E(x^k) - \nabla h(x^k) - \delta^k \mu^k p^k$. Following the same spirit in \cite{bolte2017first}, we can deduce the closed-form formula $x^{k+1} = - t^* S_{\delta^k \mu^k}(\overline{p}^k)$,
where $t^*$ is the unique positive real root of $t^3 \|S_{\delta^k \mu^k}(\overline{p}^k) \|^2 + t -1 =0$ and  $S_{\delta^k \mu^k}(\overline{p}^k)$ is the soft-thresholding operator
$S_{\delta^k \mu^k}(\overline{p}^k) := \max\{|\overline{p}^k|-\delta^k \mu^k, 0 \} {\rm sign} (\overline{p}^k).$

For any bounded set $\Omega_1\subset\RR^d$, it is easy to verify that both $\nabla h$ and $\nabla E$ are Lipschitz continuous on $\Omega_1$ and hence Assumption \ref{ass3} holds for $E$ and $h$. Assumption \ref{ass}, (ii) of Assumption \ref{ass1}, and Assumption \ref{ass2} can also be easily verified for $E$ and $h$. Besides, it is not hard to see that $h$ is strongly convex with $\dom h=\RR^d$. Therefore, we can apply Theorem \ref{mainthm2} to get the following result.

\begin{corollary}
Let the parameters $\delta^k$ and $\mu^k$ satisfy \eqref{p1}-\eqref{p2} and $\sum_{k=0}^\infty(\mu^k-\mu)<\infty$. Let $\{x^k\}$ be the sequence generated by \eqref{more2} with $x^0=p^0=0$. Then, the sequence $\{x^k\}$ has finite length.
Moreover, the sequence $\{x^k\}$ converges to a critical point $x^*$ of $E$ if $\beta=0$ and
$\lim_{k\rightarrow\infty}E(x^k)=E(x^*).$
\end{corollary}

\section{Concluding remarks}\label{se6}
The Linearized Bregman iteration was revisited in this paper from the perspective of going beyond the traditional restriction--the gradient Lipschitz continuity. In convex case, we found that the important convergence result of LBreI--Theorem \ref{Th1} still remains true even if the gradient Lipschitz continuity is replaced by the Lipschitz-like convexity condition; In nonconvex case, we were able to show global convergence under some mild assumptions. At last, we presented two examples to demonstrate the widespread application of the proposed algorithmic framework--LBreIF.

In future, we will study the convergence rate of LBreIF and consider randomized variants of LBreIF for large-scale optimization problems.

\section*{Appendix}
\begin{definition}[ Kurdyka-{\L}ojasiewicz property and function, \cite{2014Proximal}]\label{KL}
(a) The function $\sigma: \EE \rightarrow (-\infty, +\infty]$ is said to have the  Kurdyka-{\L}ojasiewicz property at $x^*\in\dom(\partial \sigma)$ if there exist $\eta\in (0, +\infty]$, a neighborhood $U$ of $x^*$ and a continuous concave function $\varphi: [0, \eta)\rightarrow\RR_+$ such that
\begin{enumerate}
  \item $\varphi(0)=0$.
  \item $\varphi$ is $C^1$ on $(0, \eta)$.
  \item for all $s\in(0, \eta)$, $\varphi^{\prime}(s)>0$.
  \item for all $x$ in $U\bigcap[\sigma(x^*)<\sigma<\sigma(x^*)+\eta]$, the Kurdyka-{\L}ojasiewicz inequality holds
  \begin{equation}\label{KL1}
  \varphi^{\prime}(\sigma(x)-\sigma(x^*))\dist(0,\partial \sigma(x))\geq 1.
\end{equation}
\end{enumerate}

(b) Proper lower semicontinuous functions which satisfy the Kurdyka-{\L}ojasiewicz inequality at each point of $\dom(\partial \sigma)$ are called KL functions.
\end{definition}

\begin{lemma}
 [Uniformized KL property, \cite{2014Proximal}] \label{UKL} Let $\Omega$ be a compact set and let $\sigma ¦Ò:\EE\rightarrow (-\infty, +\infty]$  be a proper lsc function.  Assume that $\sigma$ is constant on $\Omega$ and satisfies the KL property at each point of $\Omega$. Then, there exist $\zeta>0, \eta>0$ and $\varphi$ satisfies the same conditions as in Definition \ref{KL} such that for all $x^*\in \Omega$ and all $x$ in
 \begin{equation}
 \{x:\dist(x,\Omega)<\zeta\}\bigcap \{x: \sigma(x^*)<\sigma(x)<\sigma(x^*)+\eta\}
\end{equation}
 the condition \eqref{KL1} holds.
\end{lemma}

\section*{Acknowledgements}
The first author was supported by the National Science Foundation of China (No.11971480), the Natural Science Fund of Hunan for Excellent Youth (No.2020JJ3038), and the Fund for NUDT Young Innovator Awards (No. 20190105).

\end{document}